\RequirePackage{fix-cm}
\documentclass[smallextended]{svjour3}       % onecolumn (second format)
\smartqed  % flush right qed marks, e.g. at end of proof
\usepackage{graphicx}
\usepackage{amsmath}
\usepackage{amssymb}
\usepackage{hyperref}
\newcommand{\real}{\mathbb{R}}
\usepackage{enumerate}
\newcommand{\Diag}{\mathop{\mathrm{diag}}}

\newcommand{\calS}{{\cal S}}
\newcommand{\barx}{\bar{x}}
\newcommand{\diag}{\mathop{\mathrm{diag}}}

\renewcommand{\Re}{{\mathbb R}}
\newcounter{algo}
\newenvironment{algo}[1]{\refstepcounter{algo}  
\begin{center}
\begin{minipage}{0.9\textwidth}   \hrule\smallskip
\textbf{Algorithm \thealgo: #1}
\par\smallskip\hrule\smallskip\ignorespaces}{\par\smallskip\hrule
\end{minipage}
\end{center}
}{}

\begin{document}

\title{A Potential Reduction Method for Canonical Duality, with an Application to the Sensor Network Localization Problem%\thanks{Grants or other notes
}

\titlerunning{Potential Reduction Method for Canonical Duality}        % if too long for running head

\author{Vittorio Latorre  
}

\authorrunning{V. Latorre} % if too long for running head

\institute{V. Latorre $\cdot$\at
Department of Computer, Control and Management Engineering,
``Sapienza'', University of Rome,  Via Ariosto 25,
00185, Rome Italy\\              
              \email{latorre@dis.uniroma1.it}           %  \\
           }

\date{Received: date / Accepted: date}

\maketitle

\begin{abstract}
We propose to solve large size instances of the non-convex optimization problems reformulated with canonical duality theory. To this aim we propose an interior point potential reduction algorithm based on the solution of the primal-dual total complementarity (Lagrange) function. We establish the global convergence result for the algorithm under mild assumptions and demonstrate the method on instances of the Sensor Network Localization problem. Our numerical results are promising and show the possibility of devising efficient interior points methods for non-convex duality.

\keywords{ Interior Points Methods \and Canonical duality theory \and Global Optimization \and Sensor Localization Problem}
\end{abstract}
\section{Introduction}
We want to introduce a framework to solve % large size instances of 
the following saddle point problem:
\begin{equation}\label{eq: intsad}
\min_{x\in\real^n}\max_{\sigma\in\real^m} 
\displaystyle\Xi(x,\sigma)=\frac{1}{2} x^TG(\sigma)x-F(\sigma)^Tx-V^*(\sigma)\vspace{0.5em}, \quad s.t.\quad G(\sigma)\succeq 0
\end{equation}
Where $\succeq$ indicates that $G$ is positive semidefinite, $G(\sigma)$ is a $n\times n$ symmetric matrix such that the map $G(\sigma):\real^n\rightarrow \real^{n\times n}$ is positive semidefinite convex, that is:
$$
G(t\sigma_1+(1-t)\sigma_2)\succeq tG(\sigma_1)+(1-t)G(\sigma_2),\quad \forall \sigma_1,\sigma_2\in\real^m,\forall t\in(0,1).
$$
$V^*(\sigma)$ is a convex and two times continuosly differentiable function in $\sigma$. it is easy to notice that Problem (\ref{eq: intsad}) is convex in $x$ for every $\sigma$ such that $G(\sigma)\succeq 0$ and it is concave for every $\sigma$.

Such problem arises from the reformulation of non-convex optimization problems in Canonical Duality Theory. Canonical duality is a methodology to formulate the dual of non-convex optimization problems without any duality gap between the stationary points of the primal problem and the stationary points of the dual problem. The interest in canonical duality is not only due to the absence of duality gap, but also for the possibility to define global optimality conditions for many of such non-convex optimization problems. In the recent years, canonical duality theory has been applied
in biology,  engineering, sciences \cite{gao-cace09,wang-etal}, and recently in network communications \cite{g-r-p,ruan-gao-ep}, radial basis neural networks \cite{LaG 14}, constrained optimization \cite{lag 13}, and mechanics \cite{LSG14}.

In spite of its theoretical prowess and range of applications, there are few results regarding the numerical solution of problems formulated with canonical duality theory. In \cite{wang-etal} several mid-sized instances of the maximum cut problem are solved, to a maximum of 500 variables, with good performances in terms of speed, however no convergence result is given. A convergence result is given in \cite{WLG12}, however not only the assumptions on the convergence are rather strong, but also the reported computational experience is not exhaustive and only few small/mid sized problems are presented. In a more recent work on the application of canonical duality theory to Quasi-Variational Inequalities \cite{LS14}, the authors reformulate problem (\ref{eq: intsad}) as a monotone Variational Inequality (VI) and are able to solve high dimensional problems with several thousand of variables, without giving any convergence result, but suggesting that the methodology could have some interesting proprieties.

In this paper we partially resume the approach presented in \cite{LS14}. We consider the Karush-Kunt-Tucker conditions of the monotone variational inequality associated with (\ref{eq: intsad}), reformulate the problem as a system of constrained equations and then apply an interior point method.
Our main contributions consist in:
\begin{itemize}
\item A general interior point method for finding a solution of problem (\ref{eq: intsad}) which convergence result can be achieved under mild assumptions;
\item A numerical testing on problems with several thousands of variables, which solutions are reached efficiently.
\end{itemize}
The framework we present is quite general. However, for many non-convex problems reformulated with canonical duality, it is possible to simplify the feasible set and use linear constraints instead of matrix constraints. Such modification greatly reduces the computational burden on the algorithm. Showing that such simplifications are possible %and can be easily implemented 
in the proposed framework is another contribution of this paper.

The approach we consider is a potential reduction algorithm based on the damped Newton method reported in \cite{FaP 03} and \cite{MP99}. The framework of this algorithm rests on six main assumptions on the operator, the feasible set and the potential reduction merit function. 
The convergence result easily follows once it is proved that the proposed methodology satisfies these assumptions. The same framework has been applied to Generalized Nash Equilibrium Problems \cite{DFKS 11} and more recently to Quasi-Variational Inequalities \cite{FKS 13}, providing in both cases new important benchmarks to solve these problems.

The proposed approach is tested on instances of the Sensor Network Localization (SNL) problem \cite{BY04,BY06,KKW09,ts07}. Such problem arises in monitoring and controlling applications using wireless sensor networks such as inventory management and gathering environmental data. The problem is also related to distance geometry problems arising in predicting molecule structures and to graph rigidity. SNL consists in locating $N$ sensors of unknown position knowing the positions of $N_a$ sensors, called anchors, and the sensor-sensor, sensor-anchor distances. The methods to solve this problem are based on semidefinite programming relaxations, but in many cases these relaxations are not good enough to give a satisfying solution. Such problem can also be reformulated as a non-convex least squares optimization problem \cite{nie09}, but the presence of many local minima makes the global solution difficult to find with traditional local search approaches. 

The SNL problem was studied with canonical duality in \cite{ruan-gao-ep}, where the global optimality conditions are reported, but with limited numerical experience. The importance and the difficulty of SNL problem, especially for large sized instances, makes it suitable to test the presented framework. Furthermore, as we already mentioned, it is possible to show that the approach can be adapted to the particular proprieties of the SNL problem.

The paper is organized as follows. In the next Section we introduce a brief review on canonical duality, in order to make clear the motivations of this paper and the wide range of problems to which the proposed framework can be applied. In the first part of Section 3 we reformulate problem (\ref{eq: intsad}) as a system of equations, while in the second part we briefly report the key assumptions of the framework introduced in \cite{MP99} and present the interior point method together with its convergence proprieties and the boundedness of the generated sequence. In Section 4 we analyze the application of canonical duality to the SNL problem, and show how the general framework can be adapted to solve this problem efficiently. In Section 5 we report the numerical results on large size instances of the SNL problem. Finally in Section 6 we report the conclusions.

\emph{Notation.} For a given subset of $S$ of $\real^n$ we let int $S$, cl $S$, and bd $S$ denote, respectively, the interior, the closure and the boundary of $S$; Given a set $\cal A$ we indicate with $|{\cal A}|$ the number of elements in $\cal A$. 
 If the mapping $H:\real^n\rightarrow\real^n$ is differentiable in a point $x$ in its domain, the Jacobian matrix of $H$ at $x$ is denoted $JH(x)$. 

The set of real matrices with $n$ rows and $m$ columns is defined as $\real^{n\times m}$; the set of $n-dimensional$ squared and symmetric matrices is denoted as $\calS^n$; given a matrix $A$, we denote with $a_{ij}$ its element on the $i^{th}$ row  and $j^{th}$ column. The inner product defined on the set $	\real^{n\times n}$ of squared matrices is given by
$$
X\bullet Y= tr(X^TY),\quad (X;Y)\in \real^{n\times n},
$$
where ``tr'' denotes the trace of a matrix. This inner product induces the Frobenius norm for matrices given by
$$
\|X \|_F= \sqrt{tr(X^TX)}, \quad X \in \real^{n\times n}.
$$
Given a mapping $F(x,Y):\real^n\times\calS^n\rightarrow\real^n\times\calS^n$ defined as:
$$
F(x,Y)=
\left(\begin{array}{l}
g(x,Y)\\
h(x,Y)\\
\end{array}\right),
$$
with $g(x,Y):\real^n\times\calS^n\rightarrow\real^n$ and $h(x,Y):\real^n\times\calS^n\rightarrow\calS^n$,
 a vector $\bar{x}\in\real^n$ and a matrix $\bar{Y}\in\calS^n$, with a small abuse of notation we define the product between the mapping and the elements of $\real^n\times\calS^n$ as:
$$
F(x,Y)\bullet(\bar{x},\bar{Y})=g(x,Y)^T\bar{x}+h(x,Y)\bullet\bar{Y}.
$$
The subsets of $\calS^n$ consisting of the positive semidefinite and positive definite matrices are denoted by $\calS_+^n$ and $\calS_{++}^n$ respectively. For two matrices $A$ and $B$ in $\calS^n$, we write $A\succeq B$ if $A-B\in \calS^n_+$; similarly, $A\succ B$ means $A-B\in\calS_{++}^n$; furthermore we define $\preceq$ and $\prec$ such that  $A\preceq B$ if $-A\succeq -B$ and $A \prec B$ if $-A \succ -B$. $\Re^n_+ \subset \Re^n$ denotes the set of nonnegative numbers in $\Re^n$; $\Re^n_{++} \subset \Re^n$ denotes the set of positive numbers in $\Re^n$; sta$\{f(x) : x \in {\cal X}\}$ denotes the set of stationary points of function $f$ in ${\cal X}$;  $\diag(a)$ denotes the (square) diagonal matrix whose diagonal entries are the elements of the vector $a$;
$vect\{A\}$ denotes the vector $\in \real^{n^2}$ such that the first n elements are the elements in the first column of $A$, the elements from $n+1$ to $2n$ are the elements in the second column of $A$ and so on till the last $n$ elements that correspond to the elements in the $n^{th}$ column of $A$;
 $\circ$ denotes the Hadamard (component-wise) product operator; ${\bf 0}_n$ denotes the origin in $\real^n$, likewise ${\bf 0}_{n\times m}$ denotes the origin in $\real^{n\times m}$. If no index is indicated, the dimension of $\bf 0$ is deduced from the context;  ${\bf 1}_n$ denotes the vectors of all ones in $\real^n$; $I_n$ denotes the identity matrix in $\real^{n\times n}$.

\label{Introduction}

\section{Brief Review on Canonical Duality}
Canonical duality theory is composed mainly of 1) a canonical transformation; 2) a complementary-dual principle;
3) a triality theory.
This theory can be   demonstrated
by solving the following
general non-convex problem:
 $$
({\cal P}): \min_{x \in {\mathbb R}^n}\left\{P(x)=W(x)+\frac{1}{2} x^TAx-c^Tx \right\},
$$
where $W(x)$ is a general non-convex term in the objective function, $A\in\calS^n$ and $c\in \real^n$.
The key idea of canonical dual transformation is to choose a certain geometrically reasonable non-linear measure or operator:
$$ 
\xi=\Lambda(x):  \real^n\rightarrow{\cal E}_a\subseteq\real^m
$$
such that the non-convex functional $W(x)$ can be rewritten as
\begin{equation}\label{eq: primtra}
W(x)=V(\Lambda(x))=V(\xi): {\cal E}_a\rightarrow \real,
\end{equation}
where $V$ is a convex function in $\xi$. Consequently the primal problem can be rewritten in the following form:
$$
 \min_{x \in {\mathbb R}^n}\left\{P(x)=V(\Lambda(x))+\frac{1}{2} x^TAx-c^Tx \right\}.
$$
 As $V(\xi)$ is convex and differentiable, it is possible to apply the Legendre transformation and define its Legendre conjugate:
$$
V^*(\sigma)=sta\{\xi^T\sigma-V(\xi):\xi\in {\cal E}_a \}, \quad V^*(\sigma):\calS_a\rightarrow \real,
$$
Where $\sigma$ is the dual variable defined as
\begin{equation}\label{eq: dualmap}
\sigma=\nabla V(\xi): {{\cal E}_a}\rightarrow {{\cal S}_a\subseteq \real^m},
\end{equation}
and the feasible set $\calS_a$ is:
$$
\calS_a=\left\{\sigma: sta\{\xi^T\sigma-V(\xi):\xi\in {\cal E}_a \}< \infty \right\}.
$$
By the proprieties of the Legendre transformation, $V^*(\sigma)$ is uniquely defined and convex, furthermore the Fenchel-Young equality in convex systems holds:
\begin{equation}\label{eq: fy}
\xi^T\sigma=V(\xi)+V^*(\sigma),
\end{equation}
together with the following two relations:
$$
\sigma=\nabla V(\xi) \Leftrightarrow \xi=\nabla V^*(\sigma).
$$
%therefore the (\ref{eq: dualmap}) is invertible and by the classical definitions in canonical duality \cite{GaoBook 2000}, the differentiable function $V(\xi)$ is said to be a canonical function on its domain ${\cal E}_a$. 
By exploiting the Fenchel-Young equality, it is possible to recast the primal problem $P(x)$ as
$$
\Xi(x,\sigma)=\Lambda(x)^T\sigma-V^*(\sigma)+\frac{1}{2} x^TAx-c^Tx,
$$
which is the Total Complementarity Function in canonical duality. It is possible to show that if  the operator $\Lambda(x)$ is chosen linear, this function corresponds to the Lagrangian function, therefore $\Xi(x,\sigma)$ is also regarded as the extended Lagrangian in non-convex optimization.

In many real-world applications, the geometrically nonlinear operator $\Lambda(x)$ is usually a quadratic function, say
\begin{equation}\label{eq: quadop}
\Lambda(x)=\left\{\frac{1}{2}x^T C_k x-x^Tb_k\right\}^m: \Re^n\rightarrow {\cal E}_a\subset \Re^m.
\end{equation}
In the following we present the transformation for a general quadratic operator to simplify the exposition. However the theory can be easily extended in the case of convex and  non-convex operators  by the use of the sequential canonical dual transformation \cite{GaoBook 2000}.
With operator (\ref{eq: quadop}) the total complementarity function can be reformulated as:
\begin{equation}\label{eq: genxi}
\displaystyle\Xi(x,\sigma)=\frac{1}{2} x^TG(\sigma)x-F(\sigma)^Tx-V^*(\sigma)\vspace{0.5em},\\
\end{equation}
$$
\displaystyle G(\sigma)=A+\sum_{k=1}^m C_k\sigma_k, \quad F(\sigma)=c+\sum_{k=1}^{m}\sigma_k b_k.
$$
The dual is obtained  by exploiting the stationarity conditions of (\ref{eq: genxi}) in the primal variable:
$$
\nabla_x \Xi(x,\sigma)={\bold0}_n  \Rightarrow x= G(\sigma)^{-1} F(\sigma),
$$
and substituting the newfound value in the total complementarity function:
\begin{equation}\label{eq: dual}
P^d(\sigma)=-\frac{1}{2}F(\sigma)^TG(\sigma)^{-1}F(\sigma)-V^*(\sigma).
\end{equation}
We now report the properties of the obtained dual formulation.

\begin{theorem}\label{th: tcp} $(${\bf Complementarity-Dual Principle \cite{GaoBook 2000}}$)$
The function $P^d(\sigma)$ is canonically dual to $P(x)$ in the sense that if $\bar\sigma$ is a critical point of $P^d(\sigma)$ then the vector:
\begin{equation}\label{eq: xinsig}
\bar{x}=G^{-1}(\bar\sigma)F(\bar\sigma)
\end{equation}
is a critical point of $P(x)$ and 
$$
P(\bar{x})=\Xi(\bar{x},\bar\sigma)=P^d(\bar\sigma).
$$
Conversely, if $\bar{x}$ is a solution of $P(x)$, it must be in the form (\ref{eq: xinsig}) for a critical solution $\bar\sigma$ of $P^d(\sigma)$.
\end{theorem}
The result of this theorem clearly states that there is no duality gap between the critical points of the primal and the corresponding critical points in the dual problem, even if the primal problem is non-convex.
Theorem \ref{th: tcp} has extensive applications in nonconvex
analysis and global optimization \cite{gao-cace09}.
Note that the  feasible set $\calS_a$ is not convex, then
 in  order to identify the extremality property of the critical solutions and the global optimality conditions,
 we need to introduce the following subsets of $\calS_a$:
\[
\calS^+_a = \{\sigma\in \calS_a|\; G(\sigma) \succeq  0\}, \;\;
 \calS^-_a = \{\sigma\in \calS_a|\; G(\sigma) \prec 0\}.
\]
\begin{theorem}\label{th: triality} $(${\bf Triality Theory \cite{gao-jogo00}}$)$
Given a critical point $(\bar{x},\bar\sigma)$  of $\Xi(x,\sigma)$,
the following  three extremality conditions hold:
\begin{enumerate}
\item {\bf Global Optimum}: The critical solution $\bar{x}$ is the unique global minimizer of $P(x)$ if and only if $\bar\sigma\in\calS_a^+$ is the global maximizer of $P^d(\sigma)$ on $\calS_a^+$ i.e.
\begin{equation}\label{eq: optcon}
\min_{x\in \real^n}P(x)=P(\bar{x})=\Xi(\bar{x},\bar\sigma)=P^d(\bar\sigma)=\max_{\sigma\in\calS_a^+}P^d(\sigma).
\end{equation}
\item {\bf Local Maximum}: if $\bar\sigma\in\calS_a^-$ then $\bar\sigma$ is a local maximizer of $P^d(\sigma)$ on its neighborhood $\calS_o\subset \calS_a^-$ if and only if $\bar{x}$ is a local maximizer of $P(x)$ on its neighborhood ${\cal X}_o\in\real^n$, i. e.
$$
\max_{x\in {\cal X}_o}P(x)=P(\bar{x})=\Xi(\bar{x},\bar\sigma)=P^d(\bar\sigma)=\max_{\sigma\in\calS_o}P^d(\sigma).
$$
\item {\bf Local Minimum}: if $\bar\sigma\in\calS_a^-$ and $n=m$ then $\bar\sigma$ is a local minimizer of $P^d(\sigma)$ on its neighborhood $\calS_o\subset \calS_a^-$ if and only if $\bar{x}$ is a local minimizer of $P(x)$ on its neighborhood ${\cal X}_o\in\real^n$, i. e.
$$
\min_{x\in {\cal X}_o}P(x)=P(\bar{x})=\Xi(\bar{x},\bar\sigma)=P^d(\bar\sigma)=\min_{\sigma\in\calS_o}P^d(\sigma).
$$
\end{enumerate}
\end{theorem}
The result reported in equation (\ref{eq: optcon}) clearly shows the global optimality conditions. The original non-convex primal problem is reduced to the maximization of the dual function $P^d(\sigma)$ on the convex set $\calS_a^+$. Furthermore it easy to notice from the (\ref{eq: dual}) that the dual is concave on $\calS_a^+$, therefore the resulting problem is convex. 

The other two results in Theorem \ref{th: triality} are the conditions for two particular stationary points in the primal, that is the local maximum with the highest value of the objective function and the local minimum with the highest value of the objective function among the stationary points. Finding these stationary points can have several application in physics and chemistry \cite{GaoBook 2000}, however such issues are outside the scope of this paper.

\begin{remark}
The Total Complementarity Principle (Theorem \ref{th: tcp}) states that $\bar{x}$ is a stationary point of the primal if and only if there exists a corresponding stationary point $\bar\sigma$ in the dual and there is no duality gap between the primal and dual functions in these two points. Therefore the absence of duality gap implies the stationarity of the primal-dual solution and viceversa.

\noindent
A similar result is also given in linear and convex optimization with the Strong Duality Theorem. As a matter of facts, strong duality states that if there exists a couple of primal and dual variables $(\tilde{x}, \tilde\sigma)$ such that there is no duality gap between them, then $\tilde{x}$ is the (global minimum) solution of the primal problem and $\tilde\sigma$ is the (global maximum) solution of the dual problem. In other words Theorem \ref{th: tcp} can be considered the generalization of the Strong Duality Principle in convex optimization.
The main difference between the results of these two theorems is that in non-convex optimization stationarity does not correspond to global optimality, because of the presence of several stationary points that could be local minima, local maxima or even saddle points.

\noindent
It is also important to underline that Weak Duality does not hold on the entire primal-dual feasible set. As a matter of facts it is easy to notice that if $\bar{x}$ is the global minimum, the value of the primal function in $\bar{x}$ is smaller than the value of the dual function in any point $\hat\sigma$ that corresponds to a local stationary point in the primal. 
\noindent
However, with the results reported in Theorem \ref{th: triality}, it is possible to see that Weak Duality holds in $\calS_a^+$, that is:
$$
P(x)\ge ÿP^d(\sigma),\quad \forall x\in\real^n,\sigma\in\calS_a^+.
$$ 
The fact that both Strong and Weak Duality hold in $\calS_a^+$ naturally gives the global optimality conditions for the primal problem. Furthermore it is important to underline that Strong Duality holds on the entire feasible set, while Weak Duality holds only in a subset. 
%Therefore, Strong Duality needs ``weaker'' assumptions than Weak Duality in Canonical Duality Theory.
This is probably due to the non-convexities of the original problem.
\end{remark}

\section{Potential Reduction Algorithm for Canonical Duality}
\subsection{Reformulation of the Problem as a System of Constrained Equations}
By the results of Theorem \ref{th: triality} it is possible to find the global solution of Problem $({\cal P})$ by different approaches. One approach is to directly solve the dual formulation on $\calS_a^+$, but this method has several faults:
\begin{itemize}
\item It is necessary to calculate the inverse of matrix $G(\sigma)$ every time  the objective function is evaluated, and such operation could be necessary several times per iteration;
\item the inverse matrix operation can become even more time expensive or generate errors in the case $G(\sigma)$ is ill-conditioned or it is not full rank;
\item if the algorithm that solves the dual problem fails to converge to a good enough approximation of a stationary point, it is difficult to retrieve informations on the corresponding point in the primal problem.
\end{itemize}
For these reasons we propose a method that exploits the information available on both the primal and dual problems and search for a saddle point of the total complementarity function in $\calS_a^+$,  that is exactly the problem in the form of (\ref{eq: intsad}). 

As we said in the introduction, our approach consists in solving the following canonical saddle point problem:
\begin{equation}\label{eq: xigen}
\min_{x\in\real^n}\max_{\sigma\in\real^m}\Xi(x,\sigma)= \frac{1}{2}x^TG(\sigma)x-F(\sigma)^Tx-V^*(\sigma)\quad \mbox{s.t.}, \quad  G(\sigma)\succeq 0,
\end{equation}
by reformulating it as the problem of finding the solution of a monotone variational inequality  on a convex set \cite{FaP 03}:
\begin{equation}\label{eq: VIgen}
\Gamma(x,\sigma)=0,\quad G(\sigma)\succeq 0,
\end{equation}
where $\Gamma:\real^{n+m}\rightarrow \real^{n+m}$ is defined as:
$$
\Gamma(x,\sigma)=\left(
\begin{array}{c}
\nabla_x \Xi(x,\sigma) \\
-\nabla_\sigma \Xi(x,\sigma) \\
\end{array}\right).
$$
The operator $\Gamma$ is monotone because $\Xi(x,\sigma)$ is convex in the primal variables for $\sigma\in\calS_a^+$ and it is concave for all $\sigma \in \calS_a$ \cite{rock70}, while the set of positive definite matrices is a convex cone.
We want to find a solution of (\ref{eq: VIgen}) by solving the Karush-Kunt-Tucker (KKT) conditions associated with the problem, that is:
\begin{equation}\label{eq: mkkt}
\begin{array}{c}
\Gamma_L(x,\sigma,L)=
\left(\begin{array}{c}
\nabla_x\Xi(x,\sigma)\vspace{0.5em}\\
-\nabla_\sigma \Xi(x,\sigma)-\nabla_\sigma (L\bullet G(\sigma))
\end{array}\right)={\bold 0}_{n+m}\vspace{0.5em}\\
L\bullet G(\sigma)=0, \quad L\succeq0, \quad G(\sigma)\succeq 0, 
\end{array}
\end{equation}
Where $L\in \calS_+^n$ is the matrix of the Lagrangian multipliers. Problems can arise when searching for the solution of (\ref{eq: VIgen}) when there are KKT points located on the boundary of the feasible set. As a matter of facts, a point satisfying conditions (\ref{eq: mkkt}) with $L \neq 0$ does not correspond to a saddle point of the total complementarity function $\Xi(x,\sigma)$. In other words we are interested in KKT points which matrix of multipliers $L$ is equal to ${\bold 0}_{n\times n}$. 

To this aim, we reformulate the conditions (\ref{eq: mkkt})  as a system of Constrained Equations (CE)
and propose an interior point method specifically designed to solve this system of Constrained Equations and send the matrix of Lagrange multipliers to zero. We introduce the matrix $W\in \calS_+^n$ of slack variables and consider the $CE(H,\Omega)$ system:
\begin{equation}\label{eq: CE}
H(z)={\bf 0},\quad z=(x,\sigma,L,W)\in \Omega
\end{equation}
Where $H:\Omega\rightarrow \calS$ with $\Omega= \real^{n+m}\times\calS_+^n\times\calS_+^n$ and $\calS= \real^{n+m}\times\calS_+^n\times\calS_+^n\times\calS_+^n$, is defined as
\begin{equation}\label{eq: operatorl}
H(x,\sigma,L,W)=
\left(\begin{array}{c}
\Gamma_L(x,\sigma,L)\\
\Phi(\sigma,L,W)\\
L\\
\end{array}\right)\vspace{0.5em}\\
\end{equation}
with  $\Phi(\sigma,L,W)$ defined as:
$$
\Phi(\sigma,L,W)=\left(\begin{array}{c}
W-G(\sigma)\\
(LW+WL)/2\\
\end{array}\right).
$$
The last set of equations in (\ref{eq: operatorl}), forces the matrix of Lagrange multipliers to go to zero when the algorithm reaches convergence, assuring that the solution of $CE(\Omega,H)$ is a saddle point of (\ref{eq: xigen}).

\subsection{Key Assumptions and Convergence Result}
In this section we present the conditions which the operator $H$ and the feasible set $\Omega$ must satisfy together with a suitable potential reduction function in order to assure the convergence to a solution of the (\ref{eq: CE}). The framework we use is the same as the one presented in \cite{FaP 03} and \cite{MP99}. This framework is based on six main assumptions that we report here for convenience.

Given the set $\Omega$, operator $H$ and a potential function $p: \mbox{int } \calS\rightarrow \real$, the following assumptions must be satisfied by a potential reduction method in order to assure convergence to a solution of the $CE(\Omega,H)$. \smallskip\\
{\bf (A1)} the closed set $\Omega$ has a nonempty interior.\smallskip\\ 
{\bf (A2)} there exists a closed set $\calS\subseteq \real^{n+m}\times\calS_+^n\times\calS_+^n\times\calS_+^n$ such that
	\begin{enumerate}
	\item $\bf 0 \in \calS$;
	\item the open set $\Omega_I= H^{-1}(\mbox{int } \calS)\cap \mbox{int } \Omega$ is nonempty;
	\item the set $ H^{-1}(\mbox{int } \calS)\cap \mbox{bd } \Omega$ is empty
	\end{enumerate}
{\bf (A3)} $H$ is continuously differentiable on $\Omega_I$, and $JH(x)$ is full rank for all $x \in \Omega_I$\smallskip\\
{\bf (A4)} for every sequence $\{ u^k \}\subset \mbox{int } \calS$ such that:
$$
\begin{array}{cccc}
 \mbox{either }& \displaystyle \lim_{k\rightarrow \infty} \| u^k \|=\infty & \mbox{or} & \displaystyle  \lim_{k\rightarrow \infty} u^k=\bar{u} \in  \mbox{bd } \calS \backslash \{0\}
\end{array}
$$
we have:
$$
\lim_{k\rightarrow \infty} p(u^k)=\infty.
$$
{\bf (A5)} $p$ is continuously differentiable in its domain and $u\bullet\nabla p(u)>0$ for all nonzero $u \in \mbox{int } \calS$.\smallskip\\
{\bf (A6)} there exists a nonzero vector $o\in \calS$ and a scalar $\bar\beta\in (0,1]$ such that:
$$
u \bullet \nabla p(u)\ge \bar\beta\frac{(o\bullet u)(o\bullet \nabla p(u))}{\|o \|^2}, \quad \forall u\in \mbox{int } \calS.
$$ 
In the following theorems we show that operator $H$ and the feasible set $\Omega$ satisfy the aforementioned assumptions with the choice of a suitable potential reduction function.

\begin{theorem}
suppose that $V( \Lambda(x))$ is differentiable in $x$ and that $V^*(\sigma)$ is twice differentiable in $\sigma$, then the set $\Omega$ and the operator $H$ in (\ref{eq: operatorl}) satisfy conditions (A1)-(A3).
\end{theorem}

\begin{proof}
Condition $(A1)$ is trivially satisfied, also condition $(A2).1$ holds. The point ${(\bold 0}_{n+m},I_n,I_n)$ belongs to both $\Omega_I$ and int $\Omega$, therefore condition $(A2).2$ holds. From condition 
$$
(LW+WL)/2
$$
we can define the following set:
$$
{\cal U}=\{(L,W)\in \calS_{++}^n\times \calS_{++}^n: LW+WL\in \calS_{++}^n\}
$$
it has been proved in lemma 1 of \cite{MP98} that
$$
{\cal U}=\{(L,W)\in \calS_{+}^n\times \calS_{+}^n: LW+WL\in \calS_{++}^n\}.
$$
This alternative representation implies the $(A2).3$. Finally condition $(A3)$ is satisfied because of the assumption on $V^*(\sigma)$ and $\nabla V (\Lambda(x))$.
\qed
\end{proof}

\begin{theorem}
the potential function $p: \calS\rightarrow \real$ defined as:
\begin{equation}\label{eq: genpot}
\begin{array}{cl}
p(a,B,C,D)=& \eta \log(\|a\|^2+\| B \|_F^2+\|C \|_F^2+\|D \|_F^2)\vspace{0.5em}-\\
&\displaystyle\log(\det(B))-\log(\det(C))-\log(\det(D)),
\end{array}
\end{equation}
where $\eta\ge {2}n$, satisfies assumptions (A4)-(A6),  with $o=({\bf 0}_n,I_n,{\bf 0}_{n\times n},{\bf 0}_{n\times n})$ and $\bar\beta<1/3$
\end{theorem}

\begin{proof}
It can be easily noticed that the value of $p$ goes to $\infty$ as the sequence $\{a_k,B_k,C_k,D_k\}$ approaches the boundary of the feasible set. Considering that $\|Z \|_F=\sqrt{tr(Z^TZ)}$, then $\|Z \|_F^2$ is the sum of the squares of the $n$ eigenvalues of $Z$ and that $\det(Z)$ is the product of said eigenvalues, we have:
$$
\begin{array}{cl}
p(a,B,C,D)=&\displaystyle\eta \log\left(\sum_{i=1}^{n+m}\|a\|^2+\sum_{i=1}^{n} b_i^2 +\sum_{i=1}^{n}c_i^2+\sum_{i=1}^{n}d_i^2 \right)-\vspace{0.5em}\\
&\displaystyle\sum_{i=1}^{n}\log b_i-\sum_{i=1}^{n}\log c_i-\sum_{i=1}^{n}\log d_i
\end{array}
$$
Where $b_i=1,\dots,n$, $c_i=1,\dots,n$ and $d_i=1,\dots,n$ are the eigenvalues of $B$, $C$ and $D$ respectively. Also considering that $n \log \left(\sum_{i=1}^n u_i\right)\ge \sum_{i=1}^n \log u_i + n\log n$ it is possible to write:
$$
p(a,B,C,D)> \left(\frac{2 \eta}{3n}-1 \right)\left(\sum_{i=1}^n \log b_i+\sum_{i=1}^n \log c_i+\sum_{i=1}^n \log d_i  \right),
$$
therefore assumption (A4) is satisfied for $\eta> \frac{3}{2}n$.
If we define:
$$
\tau={\|a\|^2+\| B \|_F^2+\|C \|_F^2+\|D \|_F^2},
$$ 
it is possible to write the derivative of the potential function p as:
$$
\nabla p(a,B,C,D)=\left(
\begin{array}{c}
\displaystyle \frac{2\eta}{\tau} a \vspace{0.5 em}\\
\displaystyle \frac{2\eta}{\tau} B -B^{-1} \vspace{0.5 em}\\
\displaystyle \frac{2\eta}{\tau} C-C^{-1}\vspace{0.5 em}\\
\displaystyle \frac{2\eta}{\tau} D-D^{-1}\vspace{0.5 em}\\
\end{array}\right),
$$
we have 
$$
(a,B,C,D)\bullet \nabla p(a,B,C,D)=2\eta-3n>0,
$$
and thus Assumption (A5) holds.
For Assumption (A6), considering that $(tr Z)^2\le n \|Z \|_F^2$ and $n^2\le(trZ^{-1})(tr Z)$ we have:
$$
\begin{array}{c}
\displaystyle\frac{[\nabla p(a,B,C,D)\bullet ({\bf 0}_n,I_n,{\bf 0}_{n\times n},{\bf 0}_{n\times n})][(a,B,C,D)\bullet ({\bf 0}_n,I_n,{\bf 0}_{n\times n},{\bf 0}_{n\times n})]} {\|({\bf 0}_n,I_n,{\bf 0}_{n\times n},{\bf 0}_{n\times n}) \|^2_F} =\vspace{0.5 em}\\
\displaystyle\frac{2\eta}{n}\frac{tr (B)^2}{\tau}-\frac{tr (B^{-1})tr( B )}{n}\le \vspace{0.5 em}\\
\displaystyle \frac{2\eta}{n} \frac{tr(B)^2}{\| B\|_F^2}-\frac{tr (B^{-1})tr (B )}{n}\le \vspace{0.5 em}\\
\displaystyle 2 \eta -n<  \frac{1}{\bar\beta}(2 \eta-3 n)=\frac{1}{\bar\beta}[(a,B,C,D)\bullet \nabla p(a,B,C,D)].
\end{array}
$$
\qed
\end{proof}
We let:
$$
z=(x,\sigma, L, W),\quad\psi(z)=p(H(z)),
$$
and report the following method that follows the same scheme of the interior-point method presented in \cite{MP99}: 

\begin{algo}{CPRA: Complementarity Potential Reduction Algorithm}{} \label{alg: CPRA}
{\tt (S.0):} 
Choose $ z^0=(x^0,\sigma^0, L^0, W^0) \in  \Omega $, $ \delta_0 > 0, \gamma \in (0,1), \bar\beta<1/3, \epsilon>0$,
      and set $ k := 0 $.\\[0.8em]
   {\tt (S.1):} If $\|\Gamma(x,\sigma)\|^2<\epsilon$: STOP\\[0.8em] %If $H(u^k)=0$: STOP.\\[0.8em]
   {\tt (S.2):} Choose a scalar $\beta_k\in (0,\bar\beta)$ and find a solution $d^k=(dx^k,d\sigma^k,dL^k,dW^k)$ of the following linear least squares problem:
   $$
   \min_{d}\left\{ \frac{1}{2}\left\|Q(z^k,d)+ H(z^k)-\beta_k\frac{o^T H(z^k)}{\|o \|^2}o\right\|^2\right\}.
   $$
where:
   
   $$
   Q(z^k,d)=\left(\begin{array}{c}
   \nabla_{xx}^2\Xi(x^k,\sigma^k)dx+\nabla^2_{x\sigma}\Xi(x^k,\sigma^k)d\sigma\\
   -\nabla^2_{x\sigma}\Xi(x^k,\sigma^k)dx-\nabla^2_{\sigma\sigma}\Xi(x^k,\sigma^k)d\sigma+\nabla_{\sigma L } (L^k\bullet G(\sigma^k))dL\\
   dW-G(d\sigma)\\
   (dL)W^k+W^k(dL)+L^k(dW)+(dW)L^k\\
   dL
   \end{array}\right)
   $$
   {\tt (S.3):} find a step size  $\alpha_k$ such that
   $$
   z^k+\alpha_k d^k \in \Omega
   $$  
   and
   $$
%   \psi(u^k+\alpha_kd^k)\le \psi(u^k)+\gamma\nabla \psi(u^k)^Td^k
\begin{array}{ll}
   \psi(z^k+\alpha_kd^k)&\le \psi(z^k)+\gamma\nabla\psi(z^k) \bullet d^k
   %\big(\nabla_x \psi(z^k)^Tdx^k+\nabla_\sigma\psi(z^k)^Td\sigma^k \\
   %&+ \nabla_W\psi(z^k)\bullet dW^k+\nabla_L\psi(z^k) \bullet dL^k\big)
   \end{array}
   $$
   \\[0.8em]
   {\tt (S.4):} Set $z^{k+1} = z^k+\alpha_kd^k$, $k \leftarrow k+1 $, and go to (S.1).
\end{algo}
Algorithm \ref{alg: CPRA} is a modified, damped version of the Newton method. At Step {\tt (S.0)} the initial values of the variables and parameters are set. In order to assure the feasibility of $z^0$, it generally suffices to put a large enough positive value of $\sigma^0$, such that $G(\sigma^0)\succ0$. 
At Step {\tt (S.1)} there is the stopping criterion that assures the the final point is a good enough approximation of a stationary point of $\Xi(x,\sigma)$. 
At Step {\tt (S.2)} the modified newton direction is calculated. As the linear system is not squared, the least-squares solution to the system of equations is returned.
One of the main features of the algorithm is the presence of the vector $o$ that bends the direction toward the interior of the feasible set. 
It is important to underline that the calculated direction at every iteration is unique for Assumption $(A3)$ and always a descent direction of $\psi(\cdot)$ in $z_k$ as shown in the following theorem:
\begin{theorem}
Suppose that conditions $(A5)$ and $(A6)$ hold. Assume also that $z\in \Omega_{I}$, $d^k=(dx^k,d\sigma^k,dL^k,dW^k)\in\real^{n+m}\times\calS_+^n\times\calS_+^n$ and $\beta\in \real$ are such that
\begin{equation}\label{eq: minsq}
\begin{array}{ll}
H(z)\neq 0,  \qquad 0\le \beta < \bar\beta,\\
d^k= arg\min_d\left\{ \frac{1}{2}\left\|Q(z,d)+ H(z)-\beta_k\frac{o^T H(z)}{\|o \|^2}o\right\|^2\right\}, 
\end{array}
\end{equation}
Where $o \in \calS$ and $\bar\beta \in [0,1]$ are as in condition $(A6)$. Then $d^k$ is a descent direction for  $\psi(\cdot)$ in $z$, that is $\nabla\psi(z)\bullet d^k<0$
\end{theorem}
\begin{proof}
We introduce the following vector in $\real^{n+m+3n^2}$:
\begin{equation}
\hat{H}(z)=
\left(\begin{array}{c}
\Gamma_L(x,\sigma,L)\\
vect\{W-G(\sigma)\}\\
vect\{(LW+WL)/2\}\\
vect\{L\}\\
\end{array}\right).
\end{equation}
The Jacobian of $\hat{H}(z)$ is the following $(n+m+3n^2)\times(n+m+2n^2) $ matrix:
\begin{equation}
J\hat{H}(z)=
\left(\begin{array}{cccc}
\nabla_{xx}^2 \Xi(x,\sigma)&\nabla_{x\sigma}^2 \Xi(x,\sigma)&{\bf 0}_{n\times n^2}&{\bf 0}_{n\times n^2}\\
-\nabla_{x\sigma}^2 \Xi(x,\sigma)&\nabla_{\sigma\sigma}^2 \Xi(x,\sigma)&C^T&{\bf 0}_{m\times n^2}\\
{\bf 0}_{n^2\times n}&C&{\bf 0}_{n^2\times n^2}&I_{n^2}\\
{\bf 0}_{n^2\times n}&{\bf 0}_{n^2\times m}&W_{en}&L_{en}\\
{\bf 0}_{n^2\times n}&{\bf 0}_{n^2\times m}&I_{n^2}&{\bf 0}_{n^2\times n^2}
\end{array}\right).
\end{equation}
Where:
\begin{equation}
W_{en}=
\left(\begin{array}{cccc}
W+I_{n}w_{11}&I_{n}w_{12}&\cdots&I_{n}w_{1n}\\
I_{n}w_{21}&W+I_{n}w_{22}&\cdots&I_{n}w_{2n}\\
\vdots&\vdots&\ddots&\vdots\\
I_{n}w_{n1}&I_{n}w_{n2}&\cdots&W+I_{n}w_{nn}
\end{array}\right),
\end{equation}

\begin{equation}
L_{en}=
\left(\begin{array}{cccc}
L+I_{n}l_{11}&I_{n}l_{12}&\cdots&I_{n}l_{1n}\\
I_{n}l_{21}&L+I_{n}l_{22}&\cdots&I_{n}l_{2n}\\
\vdots&\vdots&\ddots&\vdots\\
I_{n}l_{n1}&I_{n}l_{n2}&\cdots&L+I_{n}l_{nn}
\end{array}\right),
\end{equation}
and $C\in\real^{n^2\times m}$ is $\nabla_{\sigma L} (L\bullet G(\sigma^k))^T$.
Let $u\equiv \hat{H}(z)$, if we consider 
$\hat{d}^k\in\real^{n+m+2n^2}$, solution of the following least squares problem:
\begin{equation}\label{eq: minsqx}
\hat{d}^k=arg\min_d\left\{ \frac{1}{2}\left\|(Ju)d+ u-\beta_k\frac{\hat{o}^T u}{\|\hat{o} \|^2}\hat{o}\right\|^2\right\}
\end{equation}
where $\hat{o}$ has been suitably changed from $o$ to match the dimension of $\hat{H}(z)$, it is easy to notice that $\hat{d}^k$ is equivalent to $d^k$, solution of the least squares problem in (\ref{eq: minsq}), in the following sense:
$$
\hat{d}^k=
\left(\begin{array}{c}
dx^k\\
d\sigma^k\\
vect\{dL^k\}\\
vect\{dW^k\}\\
\end{array}\right).
$$
Furthermore, if we define:
$$
\nabla \hat\psi(z)=
\left(\begin{array}{c}
\nabla_x \psi(z)\\
\nabla_\sigma\psi(z)\\
vect\{\nabla_L\psi(z)\}\\
vect\{\nabla_W\psi(z)\}\\
\end{array}\right), \quad
\nabla \hat{p}(u)=
\left(\begin{array}{c}
\nabla_x p(H(z))\\
\nabla_\sigma p(H(z))\\
vect\{\nabla_L p(H(z))\}\\
vect\{\nabla_W p(H(z))\}\\
\end{array}\right), 
$$
for the symmetry of the matrices involved in the calculations, we have:
$$
\nabla\psi(z^k)\bullet d^k=\nabla \hat\psi(z)^T\hat{d}^k , \quad \nabla\hat\psi(z)=Ju^T \nabla \hat{p}(u).
$$
Another propriety of  $\hat{d}^k$ is that it satisfies the normal equations of (\ref{eq: minsqx}):
\begin{equation}\label{eq: normal}
\hat{d}^k=	\left(Ju^TJu\right)^{-1}Ju^T\left(\beta_k\frac{\hat{o}^T u}{\|\hat{o} \|^2}\hat{o}-u	\right).
\end{equation}
Therefore, from the assumptions of the theorem and by exploiting the (\ref{eq: normal}) it is possible to obtain:
$$
\begin{array}{rcl}
\nabla \hat\psi(z)^T\hat{d}^k&=&\nabla \hat{p}(u)^T(Ju)\hat{d}^k\\
&\overset{(\ref{eq: normal})}{=}&\nabla \hat{p}(u)^TJu\left(Ju^TJu\right)^{-1}Ju^T\left(\beta_k\frac{\hat{o}^T u}{\|\hat{o} \|^2}\hat{o}-u	\right)\\
&=&\nabla \hat{p}(u)^TJuJu^{-1}(Ju^T)^{-1} Ju^T\left(\beta_k\frac{\hat{o}^T u}{\|\hat{o} \|^2}\hat{o}-u	\right)\\
&=&\nabla \hat{p}(u)^T\left(\beta_k\frac{\hat{o}^T u}{\|\hat{o} \|^2}\hat{o}-u	\right)
\le-\nabla \hat{p}(u)^Tu(1-\frac{\beta_k}{\bar\beta})\\
&=&-\nabla p(H(z))\bullet H(z) (1-\frac{\beta_k}{\bar\beta})\overset{(A5)}{<}0,
\end{array}
$$
where with $Ju^{-1}$ and $(Ju^T)^{-1}$ are the Moore Penrose  pseudo inverses of $Ju$ and $Ju^T$ respectively. The third equality derives from the propriety:
$$
(AB)^{-1}=B^{-1}A^{-1},
$$
valid for the Moore Penrose  pseudo inverse in the case we are considering (interested readers can refer to \cite{G66}). The last equality follows from the definition of $\hat{H}(z)$ and $\hat{p}(u)$.

\qed
\end{proof}
%(see Lemma 2 in \cite{MP99}). 
At step {\tt (S.3)} the potential function (\ref{eq: genpot}) is used to measure the progress of the algorithm. Finally at Step {\tt (S.4)} the value of $k$ is updated and the loop is completed.

It is possible to observe that the sequence generated by Algorithm \ref{alg: CPRA} necessarily belongs to $\Omega$. We now present the convergence result:
\begin{theorem}
Let $\{z^k\}$ be the sequence generated by Algorithm \ref{alg: CPRA}, then:
\begin{enumerate}[(a)]
\item the sequence $\{H(z^k)\}$ is bounded;\label{the: bound}
\item any accumulation point of $\{z^k \}$, if it exists, solves $CE(\Omega, H)$;\label{the: solution}
\item $\lim_{k\rightarrow \infty} H(z^k)=0$;\label{the: limzero}
\item the sequence $\{z^k\}=\{(x^k,\sigma^k,L^k,W^k)\}$ is bounded.\label{the: ombound}
\end{enumerate}
\end{theorem}
\begin{proof}
%Let ${\cal L}(\epsilon)^0=\{u\in {\cal S} : p(u)\le p(u^0), \|u \|>\epsilon \}$. From condition $(A4)$ ${\cal L}^0$ is bounded, therefore sequence $\{H_{\delta^k}(u^k) \}$ is also bounded, and this proves the (\ref{the: bound}).
The proof of statements (\ref{the: bound}) and (\ref{the: solution}) follows from Theorem 3 of \cite{MP99}.

In order to prove the (\ref{the: limzero}) we first have to prove the (\ref{the: ombound}), that is the boundedness of $\{z^k\}$. To prove the boundless of $\{z^k\}$ we have to prove the boundedness of the sequences $\{x^k\},\{\sigma^k\},\{L^k\}$ and $\{W^k\}$. The boundedness of $\{L^k\}$ is a direct consequence of the boundedness of $\{H(z^k)\}$.

To prove the boundedness of the sequences $\{x^k\}$ and $\{\sigma^k\}$ we use the operator $\Gamma$. In detail, from the (\ref{eq: dualmap}) and (\ref{eq: genxi}) we obtain:

\begin{eqnarray}
\nabla_x \Xi(x,\sigma)=G(\sigma)x-F(\sigma)\vspace{0.5em},\label{eq: xi1}\\
-\nabla_\sigma \Xi(x,\sigma)= \sigma-\nabla V(\Lambda(x))\label{eq: xi2}.
\end{eqnarray}

It is easy to see that if one of the two sequences goes to infinity while the other converges, $\|\Gamma(x^k,\sigma^k)\|\rightarrow \infty$ contradicting the (\ref{the: bound}). 

\noindent
We consider the case in which $\{x^k\}$ and $\{\sigma^k\}$ go to infinity simultaneously. if $\{x^k\}$ is unbounded, the sequence $\{\Lambda(x^k)\}$ could either converge or go to infinity. 
Considering that $V(\Lambda(x))$ is convex and differentiable in $\Lambda(x)$ that is a non-linear operator in $x$, if $\|\Lambda(x^k)\|\rightarrow \infty$, $\|\nabla_\sigma \Xi(x^k,\sigma^k)\|\rightarrow \infty$  contradicting the (\ref{the: bound}), therefore $\{\Lambda(x^k)\}$ converges. If the sequence $\{\Lambda(x^k)\}$ converges to a finite value and $\|{\sigma^k}\|\rightarrow \infty$, from the (\ref{eq: xi2}) we have $\|\nabla_\sigma \Xi(x^k,\sigma^k)\|\rightarrow \infty$ contradicting the (\ref{the: bound}), then $\{\sigma^k\}$ converges, and also $\{x^k\}$ converges. Finally if we suppose that $\{W^k\}\rightarrow\infty$,  from the boundedness of $\{\sigma^k\}$ and constraint $W-G(\sigma)$ we obtain the desired contradiction with the (\ref{the: bound}).

The (\ref{the: limzero}) is a direct consequence of conditions (\ref{the: solution}) and (\ref{the: ombound}).
\qed
\end{proof}

\section{Canonical Duality for the Sensor Network Localization Problem}

We consider the problem with $N$ sensor and $N_a$ anchors in a space of dimension $dim$. For such network, the sensor localization problem consists in locating the $n=N*dim$ unknown coordinates of the sensors that match the given distances $h$ between the sensors and the $e$ distances between the sensors and the anchors. 
Let $\rho>0$ be a radio range. A sensor $i$ is in range to another sensor $j$ if their euclidean distance $h_{ij}$ is not greater than $\rho$. If such distance is greater than $\rho$, the two sensors do not influence each other. The same reasoning is valid for a sensor-anchor pair and their distance $e_{ik}$. For this reason we introduce the following two sets:
$$
{\cal A}_h=\left\{(i,j): \|\tilde{x}_i-\tilde{x}_j \|\le \rho, i\neq j  \right\},
{\cal A}_e=\left\{(i,k): \|\tilde{x}_i-a_k \|\le \rho  \right\},
$$
Where $\tilde{x}$ indicates the real positions of the sensors and $a$ indicates the known positions of the anchors.
Therefore, the unknown location of the sensors can be found by solving the following system of equations in the variable $x$:
$$
\begin{array}{c}
\|x_i-x_j\|=h_{ij}, \quad (i,j) \in {\cal A}_h\\
\|x_i-a_k\|=e_{ik}, \quad (i,k) \in {\cal A}_e.
\end{array}
$$ 
This system of equations can be solved for small dimensional problems, however such method is not practicable when the number of sensors is large. A way to formulate the same problem is by the following non-linear least squares optimization problem \cite{nie09}:
\begin{equation}\label{eq: primal}
\min_{x \in {\mathbb R}^n}\left\{P(x)=\frac{1}{2}\sum_{(i,j)\in A_h}\left(\|x_i-x_j \|^2-h_{ij}^2\right)^2+ \frac{1}{2}\sum_{(i,k)\in A_e}\left(\|x_i-a_k \|^2-e_{ik}^2\right)^2\right\}.
\end{equation}
The value of such optimization problem is zero only if the value of $x$ corresponds to the real locations of the sensors.

It has been shown \cite{ruan-gao-ep}, that by choosing the non-linear operators:
$$
\begin{array}{c}
\xi_{ij}^h=\Lambda_{ij}(x)=\|x_i-x_j \|^2,\\
\xi_{ik}^e=\Lambda_{ik}(x)=\| x_i-a_k\|^2,
\end{array}
$$
from $\real^{n}$ into
$$
\begin{array}{c}
{\cal E}_h=\{\xi_{ij}^h \in \real^{|A_h|}:\xi_{ij}^h\ge 0 \},\\
{\cal E}_e=\{\xi_{ik}^e \in \real^{|A_e|}:\xi_{ik}^e\ge 0 \},
\end{array}
$$ 
and introducing the quadratic functions $V_h:{\cal E}_h\rightarrow \real$ and $V_e:{\cal E}_e\rightarrow \real$ such that:
$$\begin{array}{c}
V_h({\xi^h})=\frac{1}{2}\sum_{(i,j)\in A_h}( \xi_{ij}^h-h_{ij}^2)^2,\\
V_e({\xi^e})=\frac{1}{2}\sum_{(i,j)\in A_e}( \xi_{ik}^e-e_{ik}^2)^2,
\end{array}
$$
the following duality relations are invertible:
\begin{equation}\label{eq: drelation}
\begin{array}{c}
\varsigma^h_{ij}=\frac{\partial V_h(\xi^h)}{\partial \xi_{ij}^h}=\xi_{ij}^h-h_{ij}^2, \quad (i,j)\in {\cal A}_h,\vspace{0.5em}\\
\varsigma^e_{ik}=\frac{\partial V_e(\xi^e)}{\partial \xi_{ik}^e}=\xi_{ik}^e-e_{ik}^2, \quad (i,k) \in {\cal A}_k,
\end{array}
\end{equation}
where $\varsigma^h$ and $\varsigma^e$ are the dual variables. The Legendre conjugates of the two convex functions are defined by:
$$
\begin{array}{c}
V^*_h(\varsigma^h)=\sum_{(i,j) \in A_h} \frac{1}{2}(\varsigma_{ij}^h)^2+h^2_{ij}\varsigma_{ij}^h,\vspace{0.5em}\\
V^*_e(\varsigma^e)=\sum_{(i,k) \in A_e} \frac{1}{2}(\varsigma_{ik}^e)^2+e^2_{ik}\varsigma_{ik}^e.\\
\end{array}
$$
By the Fenchel-Young equality in convex programming we have:
$$
\begin{array}{c}
V_h({\xi^h})=(\xi^h)^T\varsigma^h-V^*_h(\varsigma^h),\vspace{0.5em}\\
V_e({\xi^e})=(\xi^e)^T\varsigma^e-V^*_e(\varsigma^e),\\
\end{array}
$$
and the generalized complementarity function can be written as
\begin{equation}\label{eq: xi}
\begin{array}{ccl}
\Xi(x,\varsigma^h,\varsigma^e)&=&\displaystyle\sum_{(i,j)\in {\cal A}_h}\varsigma_{ij}^h\left(\|x_i-x_j \|^2\right)+ \sum_{(i,k)\in {\cal A}_e}\varsigma_{ij}^e\left(\|x_i-a_k \|^2\right)-\vspace{0.5em}\\
&&\displaystyle V^*_h(\varsigma^h)-V^*_e(\varsigma^e)\vspace{0.5em}\\
&=&\displaystyle \frac{1}{2}x^TG(\varsigma^h,\varsigma^e)x-F(\varsigma^h)^Tx-V^*_h(\varsigma^h)-V^*_e(\varsigma^e)
\end{array}
\end{equation}
where:
\begin{eqnarray*}
F(\varsigma^e) =
\left[
\sum_{k=1}^{N_a} 2 a_{1,k} \varsigma_{1k}^e
\cdots
\sum_{k=1}^{N_a} 2 a_{dim,k} \varsigma_{1k}^e
\cdots
\sum_{k=1}^{N_a} 2 a_{1,k} \varsigma_{nk}^e
\cdots
\sum_{k=1}^{N_a} 2 a_{dim, k} \varsigma_{nk}^e
\right]^T,
\end{eqnarray*}

\begin{eqnarray}\label{eq: G}
G(\varsigma^h, \varsigma^e)= 2(\Diag(F_1(\varsigma^h))
+\Diag(F_2(\varsigma^e))
+ G_3(\varsigma^h)),
\end{eqnarray}
with
\begin{equation}
F_1(\varsigma^h) =
\left[
\begin{array}{c}
\sum_{i=1}^N \varsigma^h_{1i}\\
\vdots\\
\sum_{i=1}^N \varsigma^h_{1i}\\
\vdots\\
\sum_{i=1}^N \varsigma^h_{ni}\\
\vdots\\
\sum_{i=1}^N \varsigma^h_{ni}
\end{array}
\right],
F_2(\varsigma^e) =
\left[
\begin{array}{c}
\sum_{k=1}^{N_a} \varsigma^e_{1k}\\
\vdots\\
\sum_{k=1}^{N_a} \varsigma^e_{1k}\\
\vdots\\
\sum_{k=1}^{N_a} \varsigma^e_{nk}\\
\vdots\\
\sum_{k=1}^{N_a} \varsigma^e_{nk}
\end{array}
\right],
\end{equation}

\begin{eqnarray*}
G_3(\varsigma^1) =
\left[
\begin{array}{ccc}
-\varsigma^h_{11}I_{dim}&\cdots&-\varsigma^h_{1n}I_{dim}\\
\vdots&\vdots&\vdots\\
-\varsigma^h_{n1}I_{dim}&\cdots&-\varsigma^h_{nn} I_{dim}
\end{array}
\right],
\end{eqnarray*}
where $\varsigma^h_{ij}=0$ if $(i,j)\neq {\cal A}_h$ and  $\varsigma^e_{ik}=0$ if $(i,k)\neq {\cal A}_e$.
By exploiting the critical conditions $\nabla_x\Xi(x,\varsigma^h,\varsigma^e)=0$ we obtain the formulation of the dual problem:
\begin{equation}
P^d(\varsigma^h,\varsigma^e)= -\frac{1}{2} F(\varsigma^e)^TG(\varsigma^h,\varsigma^e)F(\varsigma^e)-V^*_h(\varsigma^h)-V^*_e(\varsigma^e).
\end{equation}
For notational convenience we make the following change of variables $\sigma=(\varsigma^h,\varsigma^e)$, $\sigma\in \calS_a \subseteq \real^m$, where $m=|{\cal A}_h|+|{\cal A}_e|$ and $V^*(\sigma)=V^*_h(\varsigma^h)+V^*_e(\varsigma^e)$. The global optimality conditions are a direct consequence of Theorem \ref{th: triality}:

\begin{theorem}\label{th: glob}
If $\bar\sigma\in\calS_a^+$ is a critical point of the canonical dual function $P^d(\sigma)$ then the vector $\barx=G^{-1}(\bar\sigma)F(\bar\sigma)$ is the global optimal solution to the primal problem $P(x)$ and
$$
\min_{x\in \real^{n}} P(\barx)=\min_{x\in \real^{n}}\max_{\sigma\in\calS_a^+}\Xi(\barx,\bar\sigma)=\max_{\sigma\in\calS_a^+} P^d(\sigma).
$$
\end{theorem}
The solution can be found by applying Algorithm 1 to solve the monotone variational inequality:
\begin{equation}\label{eq: VI}
\Gamma(x,\sigma)=0,\quad G(\sigma)\succeq 0,
\end{equation}
as explained in the previous section. However, as we said in the introduction, for the sensor network localization problem it is possible to simplify the feasible set. From the (\ref{eq: drelation}) it follows that if $(\barx, \bar\sigma)$ is a stationary point of $\Xi(x,\sigma)$ then:
$$
\begin{array}{c}
\displaystyle \bar\varsigma^h_{ij}= \|\barx_i-\barx_j \|^2-h_{ij}^2, \quad \forall (i,j)\in {\cal A}_h\vspace{0.5em},\\
\displaystyle \bar\varsigma^e_{ik}= \|\barx_i-a_k \|^2-e_{ik}^2, \quad \forall (i,k)\in {\cal A}_e.\\
\end{array}
$$
These conditions impose that the vector of dual variables $\bar\sigma$ corresponding to the global minimum of the primal problem $\barx$ is such that
$$
\bar\sigma={\bold 0}_m.
$$ 
This observation, and the fact that $\real^m_+\subset \calS_a^+$, makes possible to replace constrain $G(\sigma)\succeq 0$ by  $\sigma_i\ge0$ for all $i=1,\dots, m$ , and problem (\ref{eq: VI}) can be cast as:
\begin{equation}\label{eq: VIsim}
\Gamma(x,\sigma)=0,\quad \sigma_i\ge 0, \forall i=1,\dots,m.
\end{equation}
As we said in the introduction, for many formulations obtained by canonical duality, the constrain on the matrix $G(\sigma)$ can be simplified, therefore the general potential reduction framework can be adapted to create faster and simpler algorithms. Constrain $G(\sigma)\succeq 0$ corresponds to introduce $n^2$ constraints while only $m\le n^2$ dual variables are introduced. It is probable that no more than $m$ constrains should be introduced for problems reformulated with canonical duality. That is, in most reformulations the constraint $G(\sigma)\succeq 0$ can be replaced by simpler constrains.

One of the main difficulties of the problem (\ref{eq: VIsim}) is that the solution is located on the boundary of the dual space. As a matter of the facts the algorithm could experience slow convergence when approaching to the boundary. Therefore we change the constraints $\sigma_i\ge0,\forall i=1,\dots,m$ in $\sigma_i+\delta\ge 0,\forall i=1,\dots,m$ and decrease the value of $\delta$ by multiplying  it for a constant $\gamma_2\in(0,1)$ at every iteration, in order to assure that the algorithm converges in $\calS_a^+$.
To solve problem (\ref{eq: VIsim}) we  simplify the (\ref{eq: CE}) in the following CE system:
\begin{equation}\label{eq: CEs}
H_\delta(z)=0, \quad z=(x,\sigma,\lambda,w)\in \Omega
\end{equation}
Where $H:\Omega\rightarrow S$ with $\Omega=\real^{n+m}\times \real^{m}_{+}\times\real^{m}_{+}$ and $S=\real^{n+m}\times \real^{m}_{+}\times\real^{m}_{+}\times\real^{m}_{+}$, is defined as:
$$
H_\delta(x,\sigma,\lambda,w)=\left(\begin{array}{c}
\nabla_x \Xi(x,\sigma)\\
-\nabla_\sigma\Xi(x,\sigma)+ \lambda\\
w-\sigma+\delta\\
w\circ \lambda\\
\lambda\\
\end{array}\right),
$$
Where $\lambda\in \real^m$ are the Lagrange multipliers and $w\in\real^m$ are the slack variables associated with the constraints. System (\ref{eq: CEs}) is far more simple than system (\ref{eq: CE}) and the fact that we have to handle linear constraints instead of a matrix constrain also increases efficiency. 
Given the potential function:
$$
p(a,b,c,d)= \eta \log(\|a\|^2+\| b \|^2+\|c \|^2+\|d \|^2)-\sum_{i=1}^{m}\left(\log(b_i)+\log(c_i)+\log(d_i)\right),
$$
we let
$$
\psi_\delta(z)=p(H_\delta(z)),
$$
and report the potential reduction algorithm for canonical duality applied to the SNL problem:
\begin{algo}{CPRAS: Complementarity Potential Reduction Algorithm for Sensor Network Localization}{} \label{alg: CPRAS}
{\tt (S.0):} 
Choose $ z^0=(x^0,,\sigma^0 \lambda^0, w^0) \in  \Omega $, $ \delta_0 > 0, \gamma_1 \in (0,1), \gamma_2 \in (0,1), \epsilon>0, o=({\bold 0}_{n+m},{\bold 1}_{m},{\bold 0}_{m},{\bold 0}_{m})$,
      and set $ k := 0 $.\\[0.8em]
   {\tt (S.1):} If $\|\Gamma(x,\sigma)\|^2<\epsilon$: STOP\\[0.8em]%$H_{\delta^k}(u^k)=0$: STOP.\\[0.8em]
   {\tt (S.2):} Choose a scalar $\beta^k\in (0,1)$ and find a solution $d^k$ of the following linear least squares problem
   $$
   \min_{d^k}\left\{\frac{1}{2}\left\|JH_{\delta^k}(z^k)d^k+H_{\delta^k}(z^k)-\beta_k\frac{o^T H_{\delta^k}(z^k)}{\|o \|^2}o\right\|^2\right\}.
   $$
   {\tt (S.3):} find a step size  $\alpha_k$ such that
   $$
   z^k+\alpha_k d^k \in \Omega
   $$  
   and
   $$
   \psi_{\delta^k}(z^k+\alpha_kd^k)\le \psi_{\delta^k}(z^k)+\gamma_1\nabla \psi_{\delta^k}(z^k)^Td^k
   $$
   \\[0.8em]
   {\tt (S.4):} Set $z^{k+1} = z^k+\alpha_kd^k$, $\delta^{k+1}=\gamma_2\delta^k$, $k \leftarrow k+1 $, and go to (S.1).
\end{algo}
Algorithm \ref{alg: CPRAS} is a particular version of Algorithm \ref{alg: CPRA}, therefore it can be easily proved that it has the same convergence proprieties.

\section{Numerical Experience}

In this section we present the numerical experience of the proposed methodology to the sensor network localization problem. The problems we analyze comprehend both 2-Dimensional and 3-Dimensional large size networks. The numerical experiments are performed on a 3.4 GHz Quad core intel i7-3770 cpu with 16 GB of RAM using Matlab 7.12.0(R2011a) in a Windows 7 environment.

Through the numerical experiments we generate the sensors randomly in the unit square $[0,1]^2$ for 2-Dimensional problems or the unit cube $[0,1]^3$ for 3-Dimensional problems. The number of sensors ranges from 500, 1000, 1500, 2000 to 2500. For 2-Dimensional problems  the radio range is set to 0.5 and  the number of anchors is 4, while for 3-Dimensional problems the radio range is 1 and the number of anchors is 8. The anchors are positioned  at the corners of the unit square or the unit cube. In order to calculate the accuracy of the solution we adopt the root square distance:
$$
RMSD=\left(\frac{1}{N}\sum_{p=1}^N\|\tilde{x}_p-x^*_p \|^2\right)^{1/2},
$$
Where $\tilde{x}$ and $x^*$ indicate the real and calculated positions of the sensors respectively.
In the numerical experiments, each type of test problem is generated five times changing the random seed for the sensors locations and then tested. 

In order to better understand the results of the proposed methodology and to have a comparative benchmark, we also solved the instances of the SNL problem using the ESDP formulation proposed in \cite{BLTWY06,WZYB08} which code can be downloaded from \url{http://www.stanford.edu/~yyye/Col.html}. ESDP is an edge based semidefinite programming relaxation for the 2-Dimensional SNL problem. The code in Matlab generates the relaxed formulation and then calls SeDuMi in oder to find a solution. SeDuMi \cite{sturm99}  is a solver written in C for optimization problems with linear, quadratic and semidefinite constraints with a Matlab interface.

The parameters in Algorithm \ref{alg: CPRAS}  are: $\eta=\frac{n+4m}{2}$, $\delta^0=0.3$ with $\gamma_2=0.9$%that is lowered to $\gamma_2=0.09$ when the algorithm reaches convergence
, the initial point is the vector of all ones for the primal variables and the vector of all tens for the dual variables. The stopping parameter is set to $\epsilon=10^{-10}$, in order to have a comparative accuracy with ESDP. The parameters in ESDP are set to their standard settings while the value of the degree is set to $3$.
In Table \ref{tab: res} we report:
\begin{itemize}
\item The dimension of the problem;
\item The number n of primal variables;
\item The average number m of dual variables;
\item The average number of iterations for CPRAS to reach a solution. We remind the reader that every iteration corresponds to solving a large size least squares problem, that consists in one of the main time consuming tasks of the algorithm. The problem is solved by using the \emph{lsqr} function in Matlab;
\item the average RMDS of CPRAS on the five instances;
\item The average time in seconds of CPRAS on the five instances. 
\item the average RMDS of ESDP on the five instances;
\item The average time in seconds of ESDP on the five instances. 
\end{itemize}
\begin{table}[ht]
\begin{center}
\begin{small}
\begin{tabular}{|cc||cccc||cc|}
\hline
&&\multicolumn{4}{|c||}{CPRAS}& \multicolumn{2}{|c|}{ESDP}\\
\hline
Dim	&	n	&	m	&	Iter	&	RMDS	&	Time(s)	&	RMDS	&	Time(s)\\
\hline												
2D	&	1000	&	8639	&	14	&	9.65E-08	&	16.2	&	2.82E-07	&	40.6	\\
2D	&	2000	&	17645	&	15	&	8.51E-08	&	56.6	&	5.13E-07	&	109.4	\\
2D	&	3000	&	26643	&	17	&	7.33E-08	&	125.8	&	3.42E-07	&	160.3	\\
2D	&	4000	&	35659	&	18	&	6.78E-09	&	236.1	&	1.58E-04	&	328.1	\\
2D	&	5000	&	44664	&	18	&	1.40E-08	&	337.4	&	3.45E-05	&	370.7	\\
\hline
\end{tabular}
\end{small}
\end{center}
\caption{Results of the proposed methodology to the sensor network localization problem on 2-Dimensional instances of different sizes compared with ESDP.}\label{tab: res}
\end {table}
In Table \ref{tab: res} it is possible to notice that the algorithm is able to converge to a very accurate solution for large sized problems with almost 50000 variables in less than 350 seconds. The smallest analyzed instances with almost 10000 variables are solved in roughly 15 seconds, and the time needed to reach a solution seems to quadruple every time the number of primal variables is doubled.
 It is also interesting to notice that the number of iterations to satisfy the stopping criterion increases very slowly with the increase of the number variables and that the precision of the algorithm is not affected by the increased size of the problem. 

From the numerical results it is possible to notice that Algorithm \ref{alg: CPRAS} is capable of obtaining better accuracy than ESDP, especially for large sized instances. 
The comparison of the CPU times between the two algorithms is somewhat more problematic. As a matter of facts, CPRAS is written in Matlab, an interpreted language, and even if one of its main computational burdens, the solution of the least squares problem, is carried out rather efficiently with the function \emph{lsqr}, it is expected to go slower than SeDuMi, an algorithm written in an high efficiency code such as C. Nevertheless we also report in Table \ref{tab: res} the CPU times because they could be of interest to the readers. These results show that even in its prototypical implementation, CPRAS compares really well to ESDP with SeDuMi. In detail, the interior point algorithm is capable to reach the solution faster than the benchmark in both smaller and bigger instances.
\begin{table}[ht]
\begin{center}
\begin{small}
\begin{tabular}{|cccccc|}
\hline
Dim& n&		m&		Iter&		RMDS&			Time(s)\\
 \hline
3D	&	1500	&	11760	&	15	&	3.31E-08	&	32.5	\\
3D	&	3000	&	23762	&	18	&	2.86E-08	&	135.1	\\
3D	&	4500	&	35757	&	19	&	2.28E-08	&	280.3	\\
3D	&	6000	&	47757	&	19	&	2.00E-08	&	480.5	\\
3D	&	7500	&	59766	&	21	&	2.12E-08	&	776.8	\\
 \hline
\end{tabular}
\end{small}
\end{center}
\caption{Results of the proposed methodology to the sensor network localization problem on 3-Dimensional instances of different sizes.}\label{tab: 3d}
\end {table}

In Table \ref{tab: 3d} we report the results of CPRAS on 3-Dimensional SNL instances. The results show that Algorithm \ref{alg: CPRAS} is able to solve instances with almost 70000 variables in less than 1000 seconds. The behavior of the algorithm is quite similar to the case of 2-Dimensional instances, however it can be noticed that 3-Dimensional problems are more difficult to solve than 2-Dimensional problems, because instances with roughly the same number of variables take more iterations and time to reach the solution.

In many sensors network localization problems the measurements on the distances are affected by noise. The presence of noise greatly complicates problem (\ref{eq: primal}). Therefore we report some tests effectuated on sensor network localization problem where the distances are affected by noise in order to test the robustness of the algorithm on more complicated instances. Given a noisy factor $\alpha$, the distances are perturbed in the following way:
$$
\begin{array}{c}
\hat{d}_{ij}=\max\{(1+\alpha\nu_{ij}),0.1 \}\|\bar{x}_i-\bar{x}_j \|,(i,j)\in {\cal A}_h\\
\hat{e}_{ik}=\max\{(1+\alpha\nu_{ik}),0.1 \}\|\bar{x}_i-a_k \|,(i,k)\in {\cal A}_e
\end{array}
$$
where $\nu_{ij}$ and $\nu_{ik}$ are chosen from the standard normal distribution. 
%$\gamma_2=0.9$ is lowered to $\gamma_2=0.72$ when the algorithm reaches convergence.

In Table \ref{tab: 2dn} and \ref{tab: 3dn} we report the results on several networks of different sizes with noisy factor $\alpha=0.001$ for 2-Dimensional and 3-Dimensional problems.
\begin{table}[ht]
\begin{center}
\begin{small}
\begin{tabular}{|cc||cccc||cc|}
\hline
&&\multicolumn{4}{|c||}{CPRAS}& \multicolumn{2}{|c|}{ESDP}\\
\hline
Dim	&	n	&	m	&	Iter	&	RMDS	&	Time(s)	&	RMDS	&	Time(s)\\
\hline												
2D	&	1000	&	11366	&	16	&	2.27E-04	&	23.5	&	3.56E-04	&	27.8	\\
2D	&	2000	&	23371	&	18	&	2.28E-04	&	85.2	&	1.16E-03	&	78.9	\\
2D	&	3000	&	35371	&	19	&	2.18E-04	&	200.5	&	3.43E-04	&	120.8	\\
2D	&	4000	&	47389	&	19	&	2.17E-04	&	333.2	&	3.41E-04	&	187.8	\\
2D	&	5000	&	59399	&	20	&	2.17E-04	&	526.6	&	3.44E-04	&	239.6	\\\hline
\end{tabular}
\end{small}
\end{center}
\caption{Results of the proposed methodology to the sensor network localization problem on 2-dimensional instances of different sizes compared with ESDP with noise.}\label{tab: 2dn}
\end {table}
\begin{table}[ht]
\begin{center}
\begin{small}
\begin{tabular}{|cccccc|}
\hline
Dim& n&		m&		Iter&		RMDS&			Time(s)\\
 \hline
3D	&	1500	&	15551	&	21	&	5.36E-04	&	65.5	\\
3D	&	3000	&	31554	&	28	&	5.57E-04	&	312.8	\\
3D	&	4500	&	47551	&	29	&	5.45E-04	&	619.0	\\
3D	&	6000	&	63551	&	29	&	5.63E-04	&	1065.6	\\
3D	&	7500	&	79569	&	32	&	5.63E-04	&	1797.5	\\
 \hline
\end{tabular}
\end{small}
\end{center}
\caption{Results of the proposed methodology to the sensor network localization problem on 3-Dimensional instances of different sizes with noise.}\label{tab: 3dn}
\end {table}

From the results reported in the tables, it is evident that the presence of noise significantly increases the difficulty of the problem for Algorithm \ref{alg: CPRAS}. Not only the number of variables in the dual problem increases in respect to the noiseless problem, but also the time needed to reach a solution increases and the precision of the solution is less accurate. 

Algorithm \ref{alg: CPRAS} still compares well in respect to ESDP. From the CPU times comparison it is evident that ESDP stops very early in noise affected instances when it founds a good enough approximation of the solution. This is also reflected with its better CPU time performances even in comparison to the noiseless problem. On the other hand, the fact that the ESDP stops early affects the accuracy of the solution. As a matter of fact, Algorithm \ref{alg: CPRAS} has better performances than ESDP in terms of accuracy in all the analyzed instances.

From the results it emerges that Algorithm \ref{alg: CPRAS} reaches a good approximation of the real solution in reasonable time in the all the 100 analyzed instances, even in the ones affect with noise, therefore showing a certain level of reliability even in its prototype version in Matlab. 

\section{Conclusions}

We presented an interior points method framework for canonical duality theory that converges under mild assumptions. The framework in this paper not only has really favorable convergence proprieties, but it is also general, able to handle large sized problems efficiently with a good level of reliability and compares well in respect to the benchmark. 

In our view, these results constitute an important step for several topics in optimization. The new findings of this paper indicate that it is possible to adapt interior points methods to the problems reformulated with canonical duality. Therefore other popular interior points methods such as primal-dual methods could be used to solve problem (\ref{eq: intsad}) and find the global solution of many non-convex optimization problems efficiently.

There are also several applications that can be investigated with the presented framework. In detail, the maximum cut problem and the radial basis function neural networks problems can also be solved with canonical duality \cite{LaG 14,wang-etal}, and the proposed algorithm could be useful to find their global solutions for large sized instances. Furthermore, a more extensive analysis with high efficiency code can be performed on the SNL problem reformulated with canonical duality theory.

\section*{Acknowledgments}
The Author would like to thank Simone Sagratella for his help. Without his suggestions in the initial conception of this method, it would have been quite difficult to understand the right path to take in order to create the presented framework. The Author would also like to thank Professor Stefano Lucidi for his suggestions in improving the paper.

\end{document}